\documentclass[12pt,reqno]{amsart}

\usepackage{latexsym}
\usepackage{amssymb}
\usepackage{graphics}
\usepackage{graphicx}
\usepackage{color}

\renewcommand{\epsilon}{\varepsilon}

\newcommand{\kahler}{K\"ahler }
\newcommand{\ma}{Monge-Amp\`ere }

\newcommand{\PP}{{\mathbb P}}

\newcommand{\R}{{\mathbb R}}
\newcommand{\C}{{\mathbb C}}

\newcommand{\CP}{\C\PP}

\newcommand{\dbar}{\bar\partial}
\newcommand{\ddbar}{\partial\dbar}

\newcommand{\MA}{{\operatorname{MA}}}
\newcommand{\hilb}{{\operatorname{Hilb}}}

\newcommand{\diag}{{\operatorname{diag}}}

\renewcommand{\phi}{\varphi}

\newcommand{\acal}{\mathcal{A}}
\newcommand{\bcal}{\mathcal{B}}

\newcommand{\dcal}{\mathcal{D}}

\newcommand{\fcal}{\mathcal{F}}

\newcommand{\hcal}{\mathcal{H}}

\newtheorem{theo}{{Theorem}}[section]

\newenvironment{example}{\medskip\noindent{\it Example:\/} }{\medskip}
\newenvironment{rem}{\medskip\noindent{\it Remark:\/} }{\medskip}

\newenvironment{ques}{\medskip\noindent{\it Question:\/} }{\medskip}

\title[image of the Hilbert map]{on the image of the Hilbert map}

\author{Jingzhou Sun }
\address{Department of Mathematics, Shantou University, Shantou City, Guangdong
	515063,China} \email{jzsun@stu.edu.cn}
\thanks{The author is partially supported by NNSF of China no.11701353.}
\date{\today}

\begin{document}

	\begin{abstract}
		We talk about the image of the Hilbert map, in particular, we show the necessary and sufficient condition for the Hilbert map to be surjective.
	\end{abstract}

	\maketitle
	
	\tableofcontents
	\section{Introduction}
	
	Let $(X,L)$ be a polarized \kahler manifold of dimension $n$. The space $\hcal$, consisting of all hermitian metrics on $L$ with positive curvature, is a central playground for many analytic geometers, especially for those who are interested in the extremal metrics. Following the ideas of quantization of \kahler metrics by Donaldson and others, the Hilbert map is often considered. More precisely, let $\bcal$ denote the space of hermitian inner products on $H^0(X,L)$, the Hilbert map $\hilb:\hcal\to \bcal$ is defined as follows: $\forall h\in \hcal$,
	$$\langle s_1,s_2 \rangle_{\hilb(h)}=\int_X (s_1,s_2)_h \omega^n/n!,  $$
	where $\omega=\frac{i}{2}\Theta_h$. More generally, let $\bcal_k$ denote the space of hermitian inner products on $H^0(X,L^k)$, we have the Hilbert map $\hilb_k:\hcal\to \bcal_k$.
	In some sense, the images of the Hilbert maps $\hilb_k$ can be seen as the shadows of the complicated infinite-dimensional space $\hcal$.
	It is believed by many	\cite{Hashimoto} that when $L$ is very ample the Hilbert map $\hilb$ is surjective. But this is not always true. Actually, 
	The following arguments show that in most cases, the image of the Hilbert map is a proper subset of $\bcal$. Let $\{s_1,\cdots,s_m\}\subset H^0(X,L)$ be a fixed basis. Then $\bcal$ can be identified with the space of $m\times m$ positive definite hermitian matrices. And then $\hilb(h)$ is a hermitian matrix with entries:
	$$(\hilb(h))_{ij}=\int_X (s_i,s_j)_h \omega^n/n!$$
	Recall that a set of sections $I\subset  H^0(X,L)$ is said to generate the line bundle $L$ if $\forall p\in X$, $\exists s\in I$ such that $s(p)\neq 0$. Assume that $\{s_1,\cdots,s_{m-1}\}$ generates $L$, then
	$$M=\max_{p\in X}\frac{|s_m(p)|^2_h}{\sum_{i=1}^{m-1}|s_i(p)|^2_h}<\infty,$$
	which is independent of the metric $h$. This implies that
	$$(\hilb(h))_{mm}<M\sum_{i=1}^{m-1}(\hilb(h))_{ii}.$$
	Therefore, if $H^0(X,L)$ is more than just generating $L$, which is the case for most very ample line bundle $L$, then the Hilbert map is not surjective. Professor Hashimoto has acknowledged this misunderstanding and is preparing an erratum to \cite{Hashimoto}.
	
	\
	
Naturally, the next question is what is the image of the Hilbert map? From our previous argument, it is easy to see that we have the following constraints. Let $\{s_1,\cdots,s_l\}\subset H^0(X,L)$ be a set of sections generating $L$, and let $\{\sigma_1,\cdots,\sigma_q\}\subset H^0(X,L)$ be any finite set of sections. Then
\begin{equation}\label{max}
	M=\max_{p\in X}\frac{\sum_{i=1}^{q}|\sigma_i(p)|^2}{\sum_{i=1}^{l}|s_i(p)|^2}<\infty.
\end{equation}

So
\begin{equation}\label{constraints}
\sum_{i=1}^{q}|\sigma_i|^2_{\hilb(h)}<M\sum_{i=1}^{l}|s_i|^2_{\hilb(h)},
\end{equation}
for any $h\in \hcal$. More generally, if $\{s_1,\cdots,s_l\}$ not necessarily generate $L$, then the local representations of $\{s_1,\cdots,s_l\}$ generate an ideal sheaf $I$. Let $J$ be the ideal sheaf generated by $\{\sigma_1,\cdots,\sigma_q\}$. Then if $J\subset I$, we still have inequalities (\ref{max}) and (\ref{constraints}).
 If we regard $\bcal$ as an open subset of the space of $m\times m$ hermitian matrices which is identified with $\R^{m^2}$, then each of inequalities (\ref{constraints}) defines a half space of $\R^{m^2}$. We let $\acal\subset \bcal$ denote the subset defined by all these inequalities. Clearly, the image of the Hilbert map is contained in $\acal$. We propose the following

 \begin{ques}
 	Is $\acal$ the image of the Hilbert map?
 \end{ques}
	
	\
	
We tend to think that the answer is Yes. Our first result gives some evidence to this tendency. To state our result, we first fix some notations. Let $V\subset H^0(X,L)$ be a subspace.  Let $\bcal_V$ denote the space of hermitian inner products on $V$. Then composing the restriction, the Hilbert map gives
 $$\hilb_V: \hcal\to \bcal_V.$$
\begin{theo}\label{main}
Let $V\subset H^0(X,L)$ be a subspace that generates $L$, which is minimal in the sense that no proper subspace of $V$ generates $L$. Then $\hilb_V$ is surjective.	
\end{theo}

 \begin{rem}
 	As we discussed in the section for reduction of constraints, under the assumption on $V$ in the theorem, the constraints of the form (\ref{constraints}) are all redundant. In other words, when $\acal=\bcal$, $\acal$ is the image of the Hilbert map.
 \end{rem}

The general case becomes complicated due to the fact that there are too many constraints, which makes it is hard to decide which hermitian metric satisfies all the constraints. 
We will give an example in section \ref{sec-3} to give the readers a glance of the subtlety of the problem. This example also shows that it is not clear whether we can reduce the set of the constraints to a compact subset. So it is not clear if $\acal$ is open.

We recall that the space $\bcal$ has a natural convex structure, namely, given $H_0,H_1\in \bcal$, $tH_1+(1-t)H_0$ is in $\bcal$ for all $t\in [0,1]$. Clearly, $\acal$ is also convex.
\begin{theo}\label{thm-main}
	The image of $\hilb$ is a convex open subset of $\bcal$.
\end{theo}

Our last result describes the image of the Hilbert map from a different angle.
	
	\

Let $\bar{\bcal}$ be the space of semi-positive hermitian bilinear forms on $H^0(X,L)$, or by abuse of notation, the semi-positive hermitian inner products on $H^0(X,L)$. Given a hermitian metric $h$ on $L$, we have a map $P_h:X\to \bar{\bcal}$ defined by
$$\langle s_1,s_2 \rangle_{P_h(x)}=(s_1(x),s_2(x))_h$$	
We let $\fcal_h\subset \bar{\bcal}$ denote the closure of the subset consisting of linear combinations of the form $$\sum_{i=1}^{t}a_iP_h(x_i), t<+\infty, x_i\in X, a_i>0.$$
Clearly, $\fcal_h$ is the cone generated by the closure of the convex hull of $P_h(X)$.
It is not hard to see that $\fcal_h$ actually does not depend on the choice of $h$, so we will simply use $\fcal$. 
\begin{theo}\label{theo-points}
	The closure of the image of the Hilbert map in $\bar{\bcal}$ is $\fcal$
\end{theo}	
\begin{rem}
If the image of the Hilbert map is $\acal$, then clearly the closure of $\acal$ in $\bar{\bcal}$ is $\fcal$. Conversely, if the closure of $\acal$ in $\bar{\bcal}$ is $\fcal$, then by theorem \ref{thm-convex}, $\hilb(\hcal)$ is a convex dense subset in $\acal$, hence the image of the Hilbert map being $\acal$.
\end{rem}

	\
	
	The structure of this article is as follows. We will first prove theorem \ref{main}, then we prove theorem \ref{theo-points}. Then after some discussions of the reduction of the linear constraints, we will prove theorem \ref{thm-main}.

	\
	
	\textbf{Acknowledgements.} The author would like to thank Professors Steve Zelditch,  Yoshinori Hashimoto and Siarhei Finski for their interest in this article.

	\

	\section{proof of the main results}

We fix a \kahler form $\omega$ on $X$. Let $\phi\in C^\infty(X)$, then the \ma equation is $$\MA(\phi)=\frac{(\omega+i\ddbar \phi)^n}{\omega^n}$$
	
	We recall the famous theorem by Aubin and Yau(see \cite{AubinMA} for example)
	\begin{theo}[Aubin-Yau]
		Given $f\in C^\infty(X)$, the \ma equation
		$$\log\MA(\phi )-\phi =f $$
		has an unique solution $\phi$, which is in $C^\infty(X)$, satisfying $\omega+i\ddbar \phi>0$
	\end{theo}
	\

\

\begin{theo}\label{thm-convex}
	The image of the Hilbert map is a convex subset of $\bcal$
\end{theo}
\begin{proof}
	We fix a metric $h\in \hcal$ with $\omega=\frac{i}{2}\Theta_h$. Then $\hcal$ can be identified with the set $\hcal_h=\{\phi\in C^\infty(X)|\omega+i\ddbar \phi>0\}$. Let $H_0=\hilb(h_0)$ and $H_1=\hilb(h_1)$, with $h_0=he^{-\phi_0}$ and $h_1=he^{-\phi_1}$. For $t\in (0,1)$, we need to find $\phi_t$ such that $\omega+i\ddbar \phi_t>0$ and
	$$\int_X (s_i,s_j)_he^{-\phi_t}\frac{(\omega+i\ddbar \phi_t)^n}{n!}=(tH_1+(1-t)H_0)_{ij},$$
	namely
	$$\int_X (s_i,s_j)_he^{-\phi_t}\frac{(\omega+i\ddbar \phi_t)^n}{n!}=\int_X \frac{(s_i,s_j)_h}{n!}[te^{-\phi_1}(\omega+i\ddbar \phi_1)^n+(1-t)e^{-\phi_0}(\omega+i\ddbar \phi_0)^n]$$
	If we write $e^{f_t}\omega^n=te^{-\phi_1}(\omega+i\ddbar \phi_1)^n+(1-t)e^{-\phi_0}(\omega+i\ddbar \phi_0)^n$, then we only need to solve the \ma equation  $$\log\MA(\phi )-\phi =f_t,$$
	which, by the theorem of Aubin and Yau, has a unique solution $\phi_t\in  C^\infty(X)$ satisfying $\omega+i\ddbar \phi_t>0$. And the theorem is proved.
\end{proof}
 \
Fixing a basis $\{s_1,\cdots,s_m\}$, we denote by $\bar{\bcal}$ the closure of $\bcal$ in the space of $m\times m$ hermitian matrices, namely the space of positive semi-definite $m\times m$ matrices.

Now we prove theorem \ref{main}
\begin{proof}
	We prove that the closure of $\hilb(\hcal)$ in $\bar{\bcal}$ is $\bar{\bcal}$. Then by theorem \ref{thm-convex}, $\hilb(\hcal)$ is a convex dense subset of $\bcal$, which must be $\bcal$ itself.
	
	Clearly the closure of  $\hilb(\hcal)$ in $\bar{\bcal}$ is still convex. Given any basis $\{s'_1,\cdots,s'_m\}$, let $H_i\in \bar{\bcal}$ be defined as
	$$|s'_j|^2_{H_i}=\delta_{j}^i$$
	By unitary transformations, in order to show that the closure $\dcal$ of $\hilb(\hcal)$ in $\bar{\bcal}$ is $\bar{\bcal}$, it suffices to show that $H_i\in \dcal, 1\leq i\leq m$. We use the idea of Demailly from \cite{Demailly1993JDG}. Fix $i$, by the assumption on $V$, $\exists p_i\in X$ such that $s'_i(p_i)\neq 0$, while $s'_j(p_i)=0$ for $j\neq i$. Fix a local coordinate $(z_1,\cdots,z_n)$ in a neighborhood $U$ around $p_i$. $\forall \epsilon>0$, we can find a positive smooth function $f_\epsilon$ such that
	\begin{itemize}
		\item $\int_X f_\epsilon \omega^n=n!$
		\item $\int_{|z|<\epsilon} f_\epsilon \omega^n>n!(1-\epsilon)$
	\end{itemize}
So when $\epsilon\to 0$, $f_\epsilon \frac{\omega^n}{n!}$ converges to the delta measure $\delta_{p_i}$. Again, we can solve the \ma equation to get $\phi_\epsilon$ satisfying $\omega+i\ddbar \phi_\epsilon>0$ and $e^{-\phi_\epsilon}(\omega+i\ddbar \phi_\epsilon)^n=f_\epsilon\omega^n$. Let $a=|s'_i(p_i)|^2_h$ and $h_\epsilon=h e^{-\phi_\epsilon}$. Then as $\epsilon\to 0$, we have $$\hilb(h_\epsilon)\to aH_i.$$
By rescaling, we see that $H_i\in \dcal$. And we have finished the proof.
\end{proof}

We then prove theorem \ref{theo-points}

\begin{proof}[proof of theorem \ref{theo-points}]
We fix a hermitian metric $h$ on $L$ with corresponding \kahler form $\omega$. Then for each $x\in X$, we can approximate the point measure $\delta_x$ with volume forms $f_\epsilon \frac{\omega^n}{n!}$. Then again by the theorem by Aubin and Yau, we can find $\phi_\epsilon$ solving the equation $e^{-\phi_\epsilon}(\omega+i\ddbar \phi_\epsilon)^n=f_\epsilon \omega^n$. So $P_h(x)$ is contained in the closure of the image of $\hilb$ in $\bar{\bcal}$. Since the closure of the image of $\hilb$ is still convex, one sees that $\fcal$ is contained in the closure of the image of $\hilb$.

Conversely, $\forall h_\phi=he^{-\phi}\in \hcal$, we claim that we can find a sequence of measures $\{m_a\}$ which are all supported on sets of discrete points satisfying 
$$\lim_{a\to\infty}\int_X (s_i,s_j)_h m_a=\int_X (s_i,s_j)_he^{-\phi}w'^n,$$ for $1\leq i, j\leq m$. One way to do this is to use Shiffman-Zelditch's result in \cite{sz5}, where they proved that the normalized zero currents of sections in $H^0(X,L^k)$ almost surely converges to $\omega'=\omega+i\ddbar \phi$. When $n=1$, the zero currents of course are measures supported on sets of discrete points. When $n>1$, we can first approximate $\omega'$ with a sequence of currents represented by smooth hypersurfaces $\Sigma_b$, that is 
$$\lim_{b\to\infty}\int_{\Sigma_b} \alpha=\int_X \alpha\wedge \omega',$$
for all smooth $(n-1,n-1)$-forms $\alpha$. So $\forall \epsilon>0$, we can find a smooth hypersurface $\Sigma_b$ so that 
$$\sup_{1\leq i, j\leq m}|\int_X (s_i,s_j)_he^{-\phi}w'^n-\int_{\Sigma_b} (s_i,s_j)_he^{-\phi}w'^{n-1}|<\epsilon.$$
 Then we can use induction to find a measure $m_a$ supported on a set of discrete points in $\Sigma_b$ satisfying $$\sup_{1\leq i, j\leq m}|\int_{\Sigma_b} (s_i,s_j)_hm_a-\int_{\Sigma_b} (s_i,s_j)_he^{-\phi}w'^{n-1}|<\epsilon.$$ 
 So we can indeed approximate  $\hilb(h_\phi)$ by points in $\fcal$.
Therefore the image of $\hilb$ in $\bar{\bcal}$ is contained in $\fcal$.
\end{proof}

	\section{discussion of constraints}\label{sec-3}
	We now consider the linear constraints of the form
	$$\sum_{i=1}^{q}|\sigma_i|^2_{\hilb(h)}<M\sum_{i=1}^{l}|s_i|^2_{\hilb(h)}.$$
	The first observation is that we can assume that $(\sigma_1,\cdots,\sigma_q)$ and $(s_1,\cdots,s_l)$ are both linearly independent.
This is easy to see. For example, if $s_j\in <s_1,\cdots,s_\alpha>$, then after a unitary transformation, we can get $(s'_1,\cdots,s'_\alpha)$ linearly independent such that $\sum_{i=1}^{\alpha}|s_i|^2=\sum_{i=1}^{\alpha}|s_i'|^2$ and $s'_1=as_j$. So
$$|s_j|^2+\sum_{i=1}^{\alpha}|s_i|^2=|\sqrt{1+|a|^2}s_1'|^2+\sum_{i=2}^{\alpha}|s_i'|^2.$$
Then by induction, we can make $\sum_{i=1}^{l}|s_i|^2=\sum_{i=1}^{\alpha}|s_i'|^2$ with  $(s'_1,\cdots,s'_\alpha)$ linearly independent.

\

Assume $<s_1,\cdots,s_l>$ minimally generates $L$ or does not generate $L$.
If $<\sigma_1,\cdots,\sigma_q>\subseteq  <s_1,\cdots,s_l>$, then by theorem \ref{main}, the constraint \ref{constraints} is redundant.
Therefore we only need to consider the constraints given by the set of sections $(\sigma_1,\cdots,\sigma_q)$ such that  $<\sigma_1,\cdots,\sigma_q>$ is not contained in $<s_1,\cdots,s_l>$. Without this assumption on $<s_1,\cdots,s_l>$, it is not clear if we can ignore the constraints produced by those $\{\sigma_1,\cdots,\sigma_q\}$ satisfying $<\sigma_1,\cdots,\sigma_q>\subseteq  <s_1,\cdots,s_l>$.

\

In general, suppose that the max of $\frac{\sum_{i=1}^{q}|\sigma_i(p)|^2}{\sum_{i=1}^{l}|s_i(p)|^2}$ is attained at $p_0$.
We, by abuse of notation, let $s'_1(z)=\frac{\sum \bar{s_i}(p_0)s_i(z)}{\sqrt{\sum |s_i(p_0)|^2}}$. Of course, the vector $\frac{1}{|s_i(p)|^2}(\bar{s_1}(p_0),\cdots,\bar{s_l}(p_0))$ is independent of the choice of local frame. Then $\forall s(z)=\sum a_is_i(z)$ such that $\sum a_is_i(p_0)=0$, we have $s(p_0)=0$. Therefore, we can make an unitary transformation $(s_1,\cdots,s_l)\to (s'_1,\cdots,s'_l)$ so that $\sum |s_i(z)|^2=\sum |s'_i(z)|^2$ and $s'_i(p_0)=0$ for $i\geq 2$. So the linear constraints $\sum_{i=1}^{q}|\sigma_i|^2_{\hilb(h)}<M(|s'_1|^2_{\hilb(h)}+\sum_{i=2}^{l}\lambda_i|s'_i|^2_{\hilb(h)})$ for $\lambda_i\geq 1$ are reduced to one constraint
$$\sum_{i=1}^{q}|\sigma_i|^2_{\hilb(h)}<M\sum_{i=1}^{l}|s'_i|^2_{\hilb(h)}$$

\

Now we prove theorem \ref{thm-main}.
\begin{proof}[proof of theorem \ref{thm-main}]
	 The tangent space of $\hcal$ at a point $h$ is $C^{\infty}(X)$, and the tangent space of $\bcal$ at identity is the space of $m\times m$ hermitian matrices. We choose a basis $\{s_1,\cdots,s_m\}$ so that $\hilb(h)=I$. We then look at the tangent map of the Hilbert map
	$$(d\hilb_h(\phi))_{ij}=\int_X (s_i,s_j)_h(\Delta \phi-\phi)\frac{\omega^n}{n!}. $$
	Since the eigenvalues of $\Delta$ are nonpositive and discrete, by using eigenfunctions as orthornormal basis, one sees that $\Delta \phi-\phi=f$ has a solution for any $f\in C^{\infty}(X)$. Then we claim that the tangent map is surjective in this case. It suffices to prove that the functions $(s_i,s_j)_h$, $1\leq i,j\leq m$, are linear independent on $X$. When $n=1$, we can choose local coordinates $z$ around a point $p$ such that $z=0$. Then since the sections $(s_i)$ are linearly independent, after some linear transformations, we can assume that the vanishing orders of $s_i, 1\leq i\leq m$ at $p$ are strictly increasing as $i$ increases. Then it is easy to see that the functions $(s_i,s_j)_h$, $1\leq i,j\leq m$ are linearly independent. For $n>1$, we can choose a smooth curve passing through a point $p$ so that the sections $s_i, 1\leq i\leq m$ are still linearly independent on this curve. Then the same argument shows that the functions $(s_i,s_j)_h$, $1\leq i,j\leq m$ are linearly independent. So the tangent maps is surjective. Then the inverse function theorem shows that the image of $\hilb$ contains an open neighborhood of $I$. Since $h$ is arbitrary, we get that the image of $\hilb$ is open. 
\end{proof}

\begin{example}
	Let $X=\CP^1$ and $L=O(2)$. Then in homogeneous coordinates $Z=[Z_0,Z_1]$, we have three sections of $L$: $s_1=Z_0^2$, $s_2=\sqrt{2}Z_0Z_1$ and $s_3=Z_2^2$. Let $Q$ denote the hermitian metric on $H^0(X,L)$ that have $\{s_1,s_2,s_3\}$ as an orthornormal basis.  
	Then clearly, $Q\in \hilb(\hcal)$. 
	\begin{itemize}
		\item Let $Q_\epsilon\in \bcal$ be the metric whose matrix with $\{s_1,s_2,s_3\}$ as a basis is $\diag(1,\epsilon,1)$. We have $Q_\epsilon\in \acal$ for small $\epsilon$. To see this, one only need to mimic the proof of theorem \ref{main} to make the mass of the volume form $e^{-\phi}(\omega+\sqrt{-1}\ddbar \phi)$ to be equally concentrated around the two points $[0,1]$ and $[1,0]$. And then add a suitable scalar to $\phi$ to get $Q_\epsilon$.
		\item Let $Q'_\epsilon\in \bcal$ be the metric whose matrix with $\{s_1,s_2,s_3\}$ as a basis is $\diag(\epsilon,1,1)$. It seems that we should also have $Q_\epsilon\in \acal$ for small $\epsilon$.  But actually for $\epsilon$ small enough, we have $$Q_\epsilon\notin \hilb(\hcal).$$
		To see this, one just needs to notice that if $Q_\epsilon=\hilb(\phi)$ then the mass of the volume form $e^{-\phi}(\omega+\sqrt{-1}\ddbar \phi)$ must be concentrated in a small neighborhood of $[0,1]$. But this would imply that the norm of $s_2$ should also be small. So it is now clear that $Q'_\epsilon\notin \hilb(\hcal)$ for $\epsilon$ small enough. But actually $Q'_\epsilon\notin \acal$ for $\epsilon$ small enough. Clearly, $\{\epsilon^{-1/2}s_1,s_3\}$ generate $L$. We then see that the maximum of $\frac{|s_2|^2}{\epsilon^{-1}|s_1|^2+|s_3|^2}$ is $\sqrt{\epsilon}$. So we must have 
		$$|s_2|^2_{\hilb(h)}<\sqrt{\epsilon}(\epsilon^{-1}|s_1|^2_{\hilb(h)}+|s_3|^2_{\hilb(h)}),$$
		which is violated by $Q'_\epsilon$. 
		\item One might doubt that the difference of the two example comes from the fact that $\{s_2,s_3\}$ do not generate $L$. We can consider the basis $\{s_1,s_2+as_1,s_3\}$ instead. Then by theorem \ref{thm-main}, for $|a|$ small enough, there exists $P\in \hilb(\hcal)$ under which  
		$H^0(X,L)$ has $\{s_1,s_2+as_1,s_3\}$ as an orthornormal basis. Then $\{s_2+as_1,s_3\}$ does generate $L$. Let $P'_\epsilon\in \bcal$ be the metric whose matrix with $\{s_1,s_2+as_1,s_3\}$ as a basis is $\diag(\epsilon,1,1)$. Then we have $|s_2|_{P'_\epsilon}^2=1+|a|^2\epsilon$, which again violates the constraint violated by $Q'_\epsilon$.

	\end{itemize}
	
\end{example}
	\bibliographystyle{plain}

	\bibliography{references}
	
\end{document}